\documentclass[12pt,reqno]{amsart}
\usepackage{amsmath,amsthm,amssymb,amsfonts,amscd}
\usepackage{mathrsfs}
\usepackage{bbm}
\usepackage{bbding}
\usepackage{graphicx,latexsym}
\usepackage{hyperref}\hypersetup{colorlinks=true, linkcolor=black}
\usepackage{geometry}
\usepackage{color}
\usepackage{xcolor}
\usepackage{picture,epic}
\usepackage{tikz}

\numberwithin{equation}{section}

\theoremstyle{plain}
\newtheorem{theorem}{Theorem}[section]
\newtheorem{lemma}[theorem]{Lemma}
\newtheorem{corollary}[theorem]{Corollary}
\newtheorem{proposition}[theorem]{Proposition}

\theoremstyle{definition}
\newtheorem{definition}[theorem]{Definition}
\newtheorem{conjecture}[theorem]{Conjecture}

\theoremstyle{remark}
\newtheorem{remark}[theorem]{Remark}

\renewcommand{\Re}{\operatorname{Re}}
\renewcommand{\Im}{\operatorname{Im}}

\newcommand{\ad}{\operatorname{ad}}

\newcommand{\GL}{\operatorname{GL}}
\newcommand{\SL}{\operatorname{SL}}

\renewcommand{\mod}{\operatorname{mod}\ }

\newcommand{\dd}{\mathrm{d}}

\makeatletter
\def\@tocline#1#2#3#4#5#6#7{\relax
	\ifnum #1>\c@tocdepth 
	\else
	\par \addpenalty\@secpenalty\addvspace{#2}%
	\begingroup \hyphenpenalty\@M
	\@ifempty{#4}{%
		\@tempdima\csname r@tocindent\number#1\endcsname\relax
	}{%
		\@tempdima#4\relax
	}%
	\parindent\z@ \leftskip#3\relax \advance\leftskip\@tempdima\relax
	\rightskip\@pnumwidth plus4em \parfillskip-\@pnumwidth
	#5\leavevmode\hskip-\@tempdima
	\ifcase #1
	\or\or \hskip 1em \or \hskip 2em \else \hskip 3em \fi%
	#6\nobreak\relax
	\hfill\hbox to\@pnumwidth{\@tocpagenum{#7}}\par
	\nobreak
	\endgroup
	\fi}
\makeatother

\title{On the $L^2$ restriction norm of large level}
\author{Chengliang Guo}
\address{Mathematical Research Center \\ Shandong University \\ Jinan \\Shandong 250100 \\China}
\email{chengliang.guo@mail.sdu.edu.cn}

\date{\today}
\thanks{This work was supported by  the National Key R\&D Program of China (No. 2021YFA1000700).}

\begin{document}

	\begin{abstract}In this paper, we prove the mass equidistribution theorem restricted to  vertical geodesic segments for holomorphic Hecke newforms of large square-free level. We utilize the effective proof of QUE in the level aspect. Moreover, we study this $L^2$ mass in the full geodesics and relate it to second moments of twisted $L$-functions.
	\end{abstract}
	
	\maketitle
	

	\section{Introduction}
	In the study of the $L^2$ mass distribution of automorphic forms, the Quantum Unique Ergodicity (QUE) conjecture, famously proposed by Rudnick and Sarnak \cite{RS-QUE}, has seen substantial progress in various settings. The conjecture and its analogues in the weight and level aspects have been established in works such as \cite{Lindenstrauss-QUE,Sound-QUE,Holowinsky-Sound-QUE-holo,Nelson-QUE-level,Nelson-Pitale-Saha-QUE-level,Hu-mass-equidistribution}. Moreover, there are many interesting questions about the restriction of $L^2$ mass to thin sets \cite{GRS-nodal-domains,Young-QUE-thin-set,Humphries-equidistribution-shrinking-sets}.
	
	Although QUE problems in different aspects share similar formulations, they are often approached by very different methods. Moreover, they behave quite differently when the mass is restricted to thin sets such as vertical geodesics.
	
	Let $\mathcal{L}$ be the vertical geodesic connecting $i$ and $\infty$  in the standard fundamental domain $\mathcal{F}$ of $\SL(2,\mathbb{Z})$, where
	\[
	\mathcal{F} = \{x+iy \in \mathbb{H} : |x|\leq 1/2, x^2+y^2 \geq 1\}
	\]
	and $\mathbb{H}$ is the upper half space equipped with the length element $\dd s(z) = \frac{\sqrt{\dd x^2 + \dd y^2}}{y}$ and the area element $\dd \mu z = \frac{\dd x \dd y}{y^2}$. In the spectral aspect, \cite{GRS-nodal-domains} proved a sharp bound
	\[
	1\ll \int_{\mathcal{L}}|\phi(z)|^2\dd s(z) =  \int_{1}^{\infty}\phi^2(iy)\frac{\dd y}{y} \ll t_{\phi}^{\varepsilon}
	\]
	where $\phi$ is an even Hecke--Maass cusp form with spectral parameter $t_{\phi}$. The sharp upper bound is based on the mean value theorem of Dirichlet polynomials. By contrast, in the weight aspect, for a Hecke newform $f$ of weight $k$, the best known upper bound to date is only
	\[
	\int_{0}^{\infty}y^{k}f^2(iy)\frac{\dd y}{y} \ll k^{\frac{1}{4}+\varepsilon}
	\]
	proved in \cite{B-K-Y-fourth-norm}. As mentioned by Young, the sharp upper bound
	\[
	\int_{0}^{\infty}y^{k}f^2(iy)\frac{\dd y}{y} \ll k^{\varepsilon}
	\]
	would imply a strong subconvexity bound
	\[
	L(\frac{1}{2}, f) \ll C(f)^{1/8+\varepsilon}
	\]
	where $C(f) = k^2$ is the analytic conductor of $f$. Such a bound would be stronger than any known subconvexity bound for any $L$-function, even for the Riemann zeta function.
	
	Young also formulated several conjectures which can be viewed as QUE problems restricted to thin sets \cite{Young-QUE-thin-set}. For Eisenstein series restricted to vertical geodesic segments, he proved
	\begin{equation}
		\label{eq:geodesicMassEquidistribution}
		\lim_{T \rightarrow \infty} \frac{1}{\log(\frac{1}{4}+T^2)} \int_0^{\infty}\psi(y) |E(iy,\tfrac{1}{2}+iT)|^2 \frac{dy}{y} = \frac{6}{\pi} \int_0^{\infty} \psi(y) \frac{dy}{y}
	\end{equation}
	where $\psi$ is a smooth, compactly-supported function on $(0,\infty)$.
	In the case of cusp forms, however, the analogous equidistribution problem remains open \cite[Conjecture 1.1]{Young-QUE-thin-set}.

In this paper, we aim to study these restricted norms in the level aspect. We establish the mass equidistribution restricted to vertical geodesic segments for Hecke newforms of large square-free level.

\subsection{$L^2$ mass equidistribution in the compact segments of large level} We define $Y_{0}(q) = \Gamma_0(q)\backslash\mathbb{H} $.	Let $H_{k}(q)$ be the Hecke basis consisting of Hecke eigenforms of weight $k$ and square-free level $q$. Let $H_{k}^{*}(q)$ be the set of Hecke newforms in $H_k(q)$.
Note that $\# H_{k}^{*}(q) \sim \frac{(k-1)q}{12}$  \cite[Corollary 2.14]{L-I-S-low-lying-zeros}. For any $f \in H_{k}(q)$, there is the Fourier expansion
\[
f(z) = a_f(1)\sum_{n\geq 1}\lambda_f(n)n^{\frac{k-1}{2}}e(nz)
\]
where $\lambda_f(n)$ is the $n$-th Hecke eigenvalue with $\lambda_{f}(1) = 1$ and satisfies Deligne's bound $|\lambda_{f}(n)| \leq d(n) = \sum_{d | n}1$.	We assume that the Petersson inner product of $f \in H_{k}(q)$ is normalized by
\[
\langle f, f \rangle_q  = \int_{Y_{0}(q)}y^{k}|f(z)|^2\frac{\dd x\dd y}{y^2} = 1.
\]
In this normalization, we have
\[
|a_f(1)|^2 = \frac{2\pi^2}{q(k-1)L(1,\ad f)}\frac{(4\pi)^{k-1}}{\Gamma(k-1)}.
\] 
Moreover, the automorphic property of $\ad f$ \cite{Gelbart-Jacquet} and the non-existence of Landau--Siegel zero of $L(s,\ad f)$ \cite{Hoffstein-Lockhart} imply
\[
L(1,\ad f) \gg_k \frac{1}{\log q}.
\]

Let $\psi(z) \in C_{c}^{\infty}(Y_{0}(1)) \hookrightarrow C_{c}^{\infty}(Y_{0}(q))$. QUE in the level aspect predicts that
\begin{equation}\label{level_aspect_QUE}
	\int_{Y_{0}(q)}\psi(z)y^{k}|f(z)|^2\frac{\dd x\dd y}{y^2} \sim \frac{3}{\pi}	\int_{Y_{0}(1)}\psi(z)\frac{\dd x\dd y}{y^2}
\end{equation}
for $f \in H_{k}^{*}(q)$ as $q$ goes to infinity. Similarly, we take a smooth, compactly-supported function $\psi(y)$ on $(1,\infty)$. Let $\psi_0(z)$ be the function on $\mathcal{L}$ defined by $\psi_0(iy) = \psi(y)$. We extend $\psi_0$ to the standard fundamental domain $\mathcal{F}$  by 
\begin{equation}
	\begin{aligned}
		\Psi_{0}(z) = \left\{\begin{array}{lr}
			\psi_0(z)   ,&   z\in \mathcal{L},\\
			0 ,&   z \notin \mathcal{L}.
		\end{array}
		\right.
	\end{aligned}
\end{equation}
We then extend it to $\mathbb{H}$ by the $\SL(2,\mathbb{Z})$-invariant property, setting $\Psi(z) := \Psi_0(\gamma^{-1}z)$ for any $z \in \gamma\mathcal{F}\subset \mathbb{H}$ and $\gamma \in \SL(2,\mathbb{Z})$. Thus $\Psi(z)$ is the $\SL(2,\mathbb{Z})$-invariant extension of $\psi(y)$ to $\mathbb{H}$.

For the above $\psi(y)$, we study the $L^2$ restriction norm in the level aspect, namely the quantities
\begin{equation}\label{restricted-norm}
	\mathcal{I}(f,\Psi) := \int_{	 \bigcup_{\gamma \in  \Gamma_0(q)\backslash\Gamma(1) }\gamma\mathcal{L}}\Psi(z)y^{k}|f(z)|^2\dd s(z)
\end{equation}
and
\begin{equation}\label{restricted-norm-1}
	\mathcal{I}(f,1) := \int_{	 \bigcup_{\gamma \in  \Gamma_0(q)\backslash\Gamma(1) }\gamma\mathcal{L}}y^{k}|f(z)|^2\dd s(z)
\end{equation}
as $q$ goes to infinity. Here $f \in H_{k}^{*}(q)$ which $k$ is fixed and $q$ goes to infinity. Note that the integrals are independent of the choice of coset representatives of $\Gamma_0(q)\backslash\Gamma(1)$.  Our main result is an equidistribution statement in the level aspect.

\begin{theorem}\label{Theorem-psi-all-geodesic}Let $\psi(y) \in C_{c}^{\infty}\left((1,\infty)\right)$ and $\Psi(z)$ be the extension of $\psi(y)$ to $\mathbb{H}$. Assume that $\mathcal{I}(f,\Psi)$ is defined by \eqref{restricted-norm}. For the increasing square-free $q$, we have
	\[
	\mathcal{I}(f,\Psi) =  \frac{3}{\pi}\int_{1}^{\infty}\psi(y)\frac{\dd y}{y} + O_{\psi,k,\varepsilon}\left((\log q)^{-1/30+\varepsilon}\right)
	\]
	where $f$ runs over weight $k$, level $q$ Hecke newforms and $\int_{Y_0(q)}y^{k}|f(z)|^2\frac{\dd x\dd y}{y^2} = 1$.
\end{theorem}

We now outline the proof of Theorem \ref{Theorem-psi-all-geodesic}. First, we use a parameterization of  $ \bigcup_{\gamma \in  \Gamma_0(q)\backslash\Gamma(1) }\gamma\mathcal{L}$. By Atkin--Lehner theory, we reduce $\mathcal{I}(f,\Psi)$ to
\begin{equation}\label{Vertical-lines}
	\mathcal{I}(f,\Psi) = \int_{1}^{\infty}\psi\left(y\right)\left(y^{k}f(iy)^2 + \sum_{\substack{d \neq 1 \\d |q}}\sum_{\ell \bmod d} \left(\frac{y}{d}\right)^{k}\big|f\left(\frac{-\ell +iy}{d}\right)\big|^2\right) \frac{\dd y}{y}.
\end{equation}
Note that this integration is over the vertical lines. After directly using Fourier expansion and Mellin inversion, we obtain the shifted convolution sums decomposition
\begin{equation}\label{Shifted-convolution-sums-decomposition}
	\mathcal{I}(f,\Psi) = \mathcal{I}(f,\Psi ,0) + 	\sum_{0<|r| \leq (\log q)^{\varepsilon}}\mathcal{I}(f,\Psi ,r) +O_{\psi, k, \varepsilon}\left((\log q)^{-2026}\right) 
\end{equation}
where $\mathcal{I}(f,\Psi ,r)$ is defined by \eqref{Shifted-convolution-sum-f-Psi-r} (at least for $q$ prime).
A key observation is that
\[
\mathcal{I}(f,\Psi, r) = \langle P_{y\psi, r}, y^{k}|f|^2\rangle_q + O_{\psi,k,\varepsilon} (q^{-1+\varepsilon})
\]
where  $P_{h, r}(z)$
is the $r$-th incomplete Poincar\'e series. Then we establish Theorem \ref{Theorem-psi-all-geodesic} by following the effective proof of QUE in the level aspect (see also \cite{Nelson-QUE-level,Lester-Matomaki-Radziwill-small-scale,wattanawanichkul2024effectivecorrelationdecorrelationnewforms}). We will use Holowinsky--Soundararajan's argument.

It is an interesting question whether there is an equidistribution phenomenon restricted to geodesics in the depth aspect. Naturally, we hope the equidistribution for the general level also holds.
\begin{conjecture}Let $\psi(y) \in C_c^{\infty}((1,\infty))$ and $\Psi(z)$ be the extension of $\psi(y)$ to $\mathbb{H}$. Assume that $\mathcal{I}(f,\Psi)$ is defined by \eqref{restricted-norm}. For the increasing positive integer $q$, we have 
	\begin{equation}
		\lim_{q\rightarrow \infty}\mathcal{I}(f,\Psi) = \frac{3}{\pi}\int_{1}^{\infty}\psi(y)\frac{\dd y}{y}
	\end{equation}
	where $f$ runs over weight $k$, level $q$ Hecke newforms and $\int_{Y_0(q)}y^{k}|f(z)|^2\frac{\dd x\dd y}{y^2} = 1$.
\end{conjecture}
\begin{remark}
	Note that \cite{Hu-mass-equidistribution} shows that the standard QUE (equation (1.2)) holds for $\psi\in C_c^\infty(Y_0(p))$ at least when $p|q$ and $p$ is fixed. It would be interesting to determine whether the restricted QUE remains valid for general $\Gamma_0(p)$-invariant test functions.
\end{remark}
In the main body of the paper, we assume that $q$ is a prime for a clean proof, except in \S\ref{sec:4}. For the case of a general square-free $q$, we summarize the necessary modifications in \S\ref{sec:4}. The proof is quite similar and involves no extra difficulty.

	\subsection{Value distribution in the full geodesics and relations to twisted $L$-functions}For  the test function $\psi(y) \equiv 1$, we have the following result.
	\begin{theorem}\label{Theorem-1-all-geodesic}Assume that $\mathcal{I}(f,1)$ is defined by  \eqref{restricted-norm-1}. For the increasing prime $q$, we have
		\begin{equation}
			\mathcal{I}(f,1) =	\frac{3}{\pi}\log q + \frac{3}{\pi}\frac{L'(1,\ad f)}{L(1,\ad f)} + O_{k,\varepsilon}\left((\log q)^{\varepsilon}\right)
		\end{equation}
		where $f$ runs over weight $k$, level $q$ Hecke newforms and $\int_{Y_0(q)}y^{k}|f(z)|^2\frac{\dd x\dd y}{y^2} = 1$.
	\end{theorem}
	A similar asymptotic formula can be obtained for square-free levels $q$ (see the proof in \S \ref{sec:5}). For the purposes of our applications, we now assume that $q$ is prime. It is interesting that $\mathcal{I}(f,1)$ is closely related to the second moments of $L$-functions. We observe that
	\[
	2\mathcal{I}(f,1) = 2\int_{0}^{\infty}y^{k}|f(iy)|^2 \frac{\dd y}{y} +\sum_{0 < |\ell| \leq \frac{q-1}{2}}\int_{0}^{\infty}y^{k}|f(\frac{\ell}{q}+iy)|^2 \frac{\dd y}{y}
	\]
	from a dual property of geodesic branches (see \eqref{I-ell} and \eqref{Parseval}). By Parseval's identity, we get
	\begin{equation}\label{I-f-1-moments-L-function}
		\begin{aligned}
			2\mathcal{I}(f,1) 
			& = \sum_{0\leq|\ell|\leq \frac{q-1}{2}}\frac{|a_{f}(1)|^2}{(2\pi)^{k+1}}\int_{-\infty}^{+\infty}|\Gamma(\frac{k}{2}+it)|^2 |L(\frac{1}{2}+it, -\frac{\ell}{q}, f)|^2 \dd t
		\end{aligned}
	\end{equation}
	where $L(s , -\frac{\ell}{q} ,f)$ is the additive twisted $L$-function defined by \eqref{Additive-twisted-L-function} for $\ell \neq 0$ and $L(s, 0, f) = 2L(s,f)$. Hence, Theorem \ref{Theorem-1-all-geodesic} implies
	\begin{multline}\label{Additive-multiplicative-moments}
		\frac{1}{q}\sum_{0<|\ell|\leq \frac{q-1}{2}}\int_{-\infty}^{+\infty}|\Gamma(\frac{k}{2}+it)|^2|L(\frac{1}{2}+it, -\frac{\ell}{q}, f)|^2 \dd t \\= \left(\frac{12L(1,\ad f)}{\pi^2}\log q +\frac{12}{\pi^2}L'(1,\ad f) \right)\int_{-\infty}^{+\infty}|\Gamma(\frac{k}{2}+it)|^2\dd t + O_{k,\varepsilon}((\log q)^{\varepsilon}).
	\end{multline}
	The same result holds for the multiplicative twisted $L$-function. For $\chi$ a primitive character mod $q$, we define the multiplicative twisted $L$-function 
	\begin{equation}\label{Multiplicative-twisted}
		L(s,\chi,f) := \sum_{n \geq 1}\frac{\lambda_{f}(n)\chi(n)}{n^{s}},\quad\quad \Re(s)>1.
	\end{equation}
	\begin{corollary}\label{Corollary-twisted}Assume that $q$ is prime.  Let $f$ be a Hecke newform of weight $k$ and level $q$. Then we have
		\begin{multline}
			\frac{1}{q}{\sum_{\chi \bmod q}}^{*}\int_{-\infty}^{+\infty}|\Gamma(\frac{k}{2}+it)|^2|L(\frac{1}{2}+it, \chi,f)|^2 \dd t \\
			= \left(\frac{12L(1,\ad f)}{\pi^2}\log q +\frac{12}{\pi^2}L'(1,\ad f) \right)\int_{-\infty}^{+\infty}|\Gamma(\frac{k}{2}+it)|^2\dd t  + O_{k,\varepsilon}((\log q)^{\varepsilon}).
		\end{multline}
	\end{corollary}
	\begin{remark}Without the average over the archimedean place, these second moments of twisted $L$-functions with a power saving error term have been studied in a sequence of works \cite{B-F-K-M-twisted,B-M-second-moment-twisted,K-M-S-Kloosterman-sum,second-moments-book}. But in their works, the implied constants in the error terms depend on $f$, for instance on its level.
	\end{remark}
	\begin{remark}As mentioned by Blomer--Mili\'cevi\'c \cite[P. 455, Equ.(1.5) ]{B-M-second-moment-twisted}, it is interesting that such an average gives possibly good asymptotic formulas. Moreover, with this average, many high moments (sixth or eighth) of $L$-functions have been established (see also \cite{C-I-S-6-moment,C-L-8-moment}).
	\end{remark} 

			\vspace{2mm}

			\textbf{Plan of the paper.} The rest of this paper is organized as follows. In \S \ref{sec:2}, we reduce the $L^2$ restricted norm to the shifted convolution sums and truncate the large shifts. In \S \ref{sec:3}, we apply Holowinsky--Soundararajan's argument to prove Theorem \ref{Theorem-psi-all-geodesic} for prime levels. In \S \ref{sec:4}, we give the sketch of the proof of Theorem \ref{Theorem-psi-all-geodesic} with square-free levels. In \S \ref{sec:5}, we prove Theorem \ref{Theorem-1-all-geodesic}. In \S \ref{sec:6}, we first introduce the dual property about geodesic branches. Then we complete the proof of Corollary \ref{Corollary-twisted}. 
			
			\textbf{Notation.} Throughout the paper, $\varepsilon$ is an arbitrarily small positive number; all of them
			may be different at each occurrence. As usual, $e(x) = e^{2\pi ix}$. We use the standard Landau and Vinogradov notations \( O(\cdot) \), \( o(\cdot) \), \(\ll\),  \(\gg\), $\asymp$ and $\sim$. Specifically, we express \( X \ll Y \), \( X = O(Y) \), or \( Y \gg X \) when there exists a constant \( C \) such that \( |X| \leq C|Y| \). If the constant $C=C_s$ depends on some object $s$, we write $X=O_s(Y)$. 
			We use \( X \asymp Y \) to denote that $c_1Y\leq X\leq c_2Y$ for some positive constants $c_1, c_2$.

			\section{\label{sec:2}Reduction to the shifted convolution sums}
			In this section, we assume that $q$ is prime. For $\psi(y) \in C_{c}^{\infty}((1,\infty))$, we study the integral
			\[
			\mathcal{I}(f,\Psi) = \int_{	 \bigcup_{\gamma \in  \Gamma_0(q)\backslash\Gamma(1) }\gamma\mathcal{L}}\Psi(z)y^{k}|f(z)|^2\dd s(z).
			\]

			\subsection{Parameterization and Fourier expansion}
			We choose the special coset representatives
			\begin{equation}\label{coset-representatives}
				\Gamma_0(q)\backslash\Gamma(1) = \big\{\begin{pmatrix}
					1 &0\\
					0 &1
				\end{pmatrix}, \begin{pmatrix}
					0 &-1\\
					1 &-\ell
				\end{pmatrix}, 0\leq|\ell|\leq \frac{q-1}{2}\big\}.
			\end{equation}
			Let $\mathcal{L}_\ell = \begin{pmatrix}
				0 &-1\\
				1 &-\ell
			\end{pmatrix}\mathcal{L}$. For $z = iy \in \mathcal{L}$, we have $\begin{pmatrix}
				0 &-1\\
				1 &-\ell
			\end{pmatrix}iy = \frac{\ell}{\ell^2 +y^2}+ \frac{y}{\ell^2 + y^2}i$. This yields a parameterization of $\mathcal{L}_\ell$. Consequently, we have
			\begin{equation}
				\mathcal{I}(f,\Psi)  = \mathcal{I}'(f,\Psi) + \sum_{0\leq |\ell| \leq \frac{q-1}{2}}\mathcal{I}_{\ell}(f, \Psi)
			\end{equation}
			where
			\[
			\mathcal{I}'(f,\Psi) = \int_{1}^{\infty}\psi_0(iy)y^{k}f(iy)^2 \frac{\dd y}{y}
			\]
			and
			\begin{equation}
				\begin{aligned}
					\mathcal{I}_{\ell}(f, \Psi) &= \int_{1}^{\infty}\Psi\left(\frac{\ell}{\ell^2 +y^2}+i\frac{y}{\ell^2 + y^2}\right)\left(\frac{y}{\ell^2 + y^2}\right)^{k}\big|f\left(\frac{\ell}{\ell^2 +y^2}+ i\frac{y}{\ell^2 + y^2}\right)\big|^2 \frac{\dd y}{y}\\
					& = \int_{1}^{\infty}\psi_0\left(iy\right)\left(\frac{y}{\ell^2 + y^2}\right)^{k}\big|f\left(\frac{\ell}{\ell^2 +y^2}+ i\frac{y}{\ell^2 + y^2}\right)\big|^2 \frac{\dd y}{y}.
				\end{aligned}
			\end{equation}	 
			By Atkin--Lehner theory \cite{Atkin-Lehner-theory} (see also \cite[Equ. (6.68)]{Iwaniec-Topics}), we get
			\[
			q^{-k/2}z^{-k}f\left(-\frac{1}{qz}\right) = \pm f(z).
			\]
			Then we obtain
			\[
			\left(\frac{y}{\ell^2 + y^2}\right)^{k}\big|f\left(\frac{\ell}{\ell^2 +y^2}+ i\frac{y}{\ell^2 + y^2}\right)\big|^2 = q^{-k}y^{k}\big|f(\frac{-\ell +i y}{q})\big|^2.
			\]
			Then we have	
			\begin{equation}
				\mathcal{I}(f,\Psi) = \int_{1}^{\infty}\psi\left(y\right)\left(y^{k}f(iy)^2 + \sum_{0\leq |\ell| \leq \frac{q-1}{2}} \left(\frac{y}{q}\right)^{k}\big|f\left(\frac{-\ell +iy}{q}\right)\big|^2\right) \frac{\dd y}{y}.
			\end{equation}
			By the Fourier expansion, we have
			\begin{equation}\label{Fourier-expansion-I-f-Psi}
				\begin{aligned}
					\mathcal{I}(f,\Psi)& = |a_f(1)|^2 \sum_{n\geq 1}\sum_{m \geq 1}\lambda_{f}(n)\lambda_{f}(m)(nm)^{\frac{k-1}{2}}\\
					&\quad\times \int_{1}^{\infty}\psi(y)\left(y^{k}e^{-2\pi (n+m)y} +\sum_{0\leq |\ell| \leq \frac{q-1}{2}}\left(\frac{y}{q}\right)^{k}e^{-\frac{2\pi (n+m)y}{q}}e^{-\frac{2\pi i (n-m)\ell}{q}} \right)\frac{\dd y}{y}\\
					& = |a_f(1)|^2 \sum_{n\geq 1}\sum_{m \geq 1}\lambda_{f}(n)\lambda_{f}(m)(nm)^{\frac{k-1}{2}}\\
					&\quad\times \int_{1}^{\infty}\psi(y)\left(y^{k}e^{-2\pi (n+m)y} +\left(\frac{y}{q}\right)^{k}e^{-\frac{2\pi (n+m)y}{q}}q\delta\left(n \equiv m \bmod q\right) \right)\frac{\dd y}{y}.
				\end{aligned}
			\end{equation}
			\begin{remark}
				Essentially, we have the following bound
				\[
				\mathcal{I}'(f,\Psi) \ll q^{-1}L(1,\ad f)^{-1} \quad\quad	\mathcal{I}_{\ell}(f,\Psi) \ll \min\{q^{\varepsilon} , \frac{\ell^2}{q^{1-\varepsilon}}\} 
				\]
				by using Deligne's bound trivially. Then $\mathcal{I}(f,\Psi) \ll q^{1+\varepsilon}$. The same bound also follows from the trivial sup norm bound $\|y^{k/2}f\|_{\infty} \ll q^{\varepsilon}$.
				
				For $r \in \mathbb{Z}$, we define 
				\begin{equation}\label{Shifted-convolution-sum-f-Psi-r}
					\begin{aligned}
						\mathcal{I}&(f,\Psi ,r) =\\	&q|a_{f}(1)|^2\sum_{n\geq 1}\sum_{m \geq 1}\lambda_f(n)\lambda_f(m)(nm)^{\frac{k-1}{2}}\int_{0}^{\infty}\psi(y)\left(\frac{y}{q}\right)^{k}e^{-\frac{2\pi(n+m)y}{q}}\delta(n-m = qr)\frac{\dd y}{y}.
					\end{aligned}
				\end{equation}
			\end{remark}
			\subsection{The shifted convolution sums}
			By Mellin inversion, we have
			\[
			y^{k}e^{-2\pi(n+m)y} = \frac{1}{2\pi i}\int_{(3)}\left(2\pi(n+m)\right)^{-k-s}\Gamma(k+s)y^{-s}\dd s. 
			\]
			Hence
			\begin{equation}
				\begin{aligned}
					&\mathcal{I}(f,\Psi) = |a_f(1)|^2 \sum_{n\geq 1}\sum_{m \geq 1}\lambda_{f}(n)\lambda_{f}(m)(nm)^{\frac{k-1}{2}}\\
					&\times  \frac{1}{2\pi i}\int_{(3)}\left(2\pi(n+m)\right)^{-k-s}\Gamma(k+s)\int_{1}^{\infty}\psi(y)\left(y^{-s} +\left(\frac{y}{q}\right)^{-s} q\delta\left(n\equiv m \bmod q\right) \right)\frac{\dd y}{y}\dd s.
				\end{aligned}
			\end{equation}
			
			The contribution of the term without the $\delta$ symbol is bounded by $O_{\psi}(q^{-1+\varepsilon})$ by a trivial application of Deligne's bound. We also have 
			\begin{equation}
				\mathcal{I}(f,\Psi ,r) 
				= \frac{1}{L(1,\ad f)} \sum_{n \geq 1}\frac{\lambda_{f}(n)\lambda_{f}(n +qr)\left(\frac{2\sqrt{n(n+qr)}}{2n+qr}\right)^{k-1}}{2n+qr}V(2n + qr, q)
			\end{equation}
			where
			\[
			V(t,q) = \frac{1}{2i}\int_{(3)}\widetilde{\psi}(-s)(\frac{2\pi t}{q})^{-s}\frac{\Gamma(s+k)}{\Gamma(k)} \dd s.
			\]
			Here  $\widetilde{\psi}(s) = \int_{1}^{\infty}\psi(y)y^{s}\frac{\dd y}{y}$ and we have
			\[
			V(t,q) \ll_{\psi,k}  \left(1 + \frac{t}{q}\right)^{-A}
			\]
			for any $A \geq 1$. Therefore, we get
			\[
			\mathcal{I}(f,\Psi) = 	\sum_{r\in \mathbb{Z}}\mathcal{I}(f,\Psi ,r) +O_{\psi,\varepsilon}(q^{-1+\varepsilon}). 
			\]

			\subsection{Estimates of $\mathcal{I}(f,\Psi ,r)$ for large $r$}
			By Deligne's bound, we get
			\begin{equation}\label{Bounds-I-f-psi-trivial}
				\begin{aligned}
					\mathcal{I}(f,\Psi ,r)
					& \ll_{\psi,k}\frac{1}{L(1,\ad f)} \sum_{n \geq 1}\frac{d(n)d(n +qr)}{2n+qr} \left(1 + \frac{2n+qr}{q}\right)^{-A}\\
					& \ll_{\psi,k}\frac{1}{L(1,\ad f)}\left(\sum_{n \leq \frac{qr}{2}}\frac{d(n)d(n +qr)}{qr} r^{-A} + \sum_{n \geq \frac{qr}{2}}\frac{d(n)d(n +qr)}{n} \left( \frac{n}{q}\right)^{-A}\right)\\
					& \ll_{\psi,k} r^{-A}(\log qr)^{6}.
				\end{aligned}
			\end{equation}
			Then for $r \geq (\log q)^{\varepsilon}$, we get
			\[
			\mathcal{I}(f,\Psi ,r) \ll_{\psi,k}(r \log q)^{-2026}
			\]
			provided $A$ is chosen sufficiently large. Thus, we obtain
			\begin{equation}
				\mathcal{I}(f,\Psi) = \mathcal{I}(f,\Psi ,0) + 	\sum_{0<|r| \leq (\log q)^{\varepsilon}}\mathcal{I}(f,\Psi ,r) +O_{\psi, k, \varepsilon}\left((\log q)^{-2026}\right). 
			\end{equation}
			\section{\label{sec:3}Applying Holowinsky--Soundararajan's argument}
			\subsection{Incomplete Eisenstein series and incomplete Poincar\'e series}
			Let
			\[
			\Gamma_{\infty} := \big\{\begin{pmatrix}
				1 & n\\
				0 &1
			\end{pmatrix}, n \in \mathbb{Z}\big\}.
			\]
			For $\psi(y) \in C_{c}^{\infty}\left((1,\infty)\right)$, we define the $r$-th incomplete Poincar\'e series by
			\[
			P_{\psi , r}(z) := \frac{1}{2}\sum_{\gamma \in \Gamma_{\infty}\backslash\SL(2,\mathbb{Z})}\psi(\Im\gamma z)e^{-2\pi i r \Re(\gamma z)}.
			\]
			When $r = 0$, we call it the incomplete Eisenstein series
			\[
			E_{\psi}(z) := P_{\psi, 0}(z) = \frac{1}{2}\sum_{\gamma \in \Gamma_{\infty}\backslash\SL(2,\mathbb{Z})}\psi(\Im\gamma z).
			\]
			The key observation is the following lemma.
			\begin{lemma}\label{Lemma-incomplete-Poincare-series}Let $\psi(y) \in C_{c}^{\infty}\left((1,\infty)\right)$ and $\Psi(z)$ be the extension to $\mathbb{H}$ defined previously. For any $\varepsilon >0$ and $r\in\mathbb{Z}$, we have
				\begin{equation}
					\mathcal{I}(f,\Psi ,r) = \int_{Y_{0}(q)}P_{y\psi ,r}(z)y^{k}|f(z)|^{2}\frac{\dd x\dd y}{y^2} + O_{k,\varepsilon}\left(q^{-1+\varepsilon}\right).
				\end{equation}	
				In particular, we have
				\begin{equation}
					\mathcal{I}(f,\Psi ,0) = \int_{Y_{0}(q)}E_{y\psi}(z)y^{k}|f(z)|^{2}\frac{\dd x\dd y}{y^2} + O_{k,\varepsilon}\left(q^{-1+\varepsilon}\right).
				\end{equation}	
			\end{lemma}
			\begin{proof}By the unfolding trick, we have
				\begin{equation}\label{Unfolding-trick}
					\int_{Y_{0}(q)}P_{y\psi ,r}(z)y^{k}|f(z)|^{2}\frac{\dd x\dd y}{y^2}
					= \sum_{d | q}\int_{0}^{\infty}dy\psi(dy)\int_{0}^{1}e^{-2\pi i rd x}y^{k}|f(x+iy)|^2 \dd x \frac{\dd y}{y^2}.
				\end{equation}
				The contribution of $d = 1$ is bounded by $\int_{1}^{\infty}\psi(y)\int_{0}^{1}y^{k}|f(x+iy)|^2 \dd x \frac{\dd y}{y} = O_{k, \varepsilon}(q^{-1+\varepsilon})$. The remainder is
				\begin{equation}
					\begin{aligned}
						q&\int_{0}^{\infty}\psi(qy)\int_{0}^{1}e^{-2\pi i rq x}y^{k}|f(x+iy)|^2 \dd x \frac{\dd y}{y}\\
						& =q|a_{f}(1)|^2\sum_{n\geq 1}\sum_{m \geq 1}\lambda_f(n)\lambda_f(m)(nm)^{\frac{k-1}{2}}\int_{0}^{\infty}\psi(qy)y^{k}e^{-2\pi(n+m)y}\delta\left(n-m = qr\right)\frac{\dd y}{y}\\
						& = q|a_{f}(1)|^2\sum_{n\geq 1}\sum_{m \geq 1}\lambda_f(n)\lambda_f(m)(nm)^{\frac{k-1}{2}}\int_{0}^{\infty}\psi(y)\left(\frac{y}{q}\right)^{k}e^{-\frac{2\pi(n+m)y}{q}}\delta\left(n-m = qr\right)\frac{\dd y}{y}.
					\end{aligned}
				\end{equation} 
				Recalling \eqref{Shifted-convolution-sum-f-Psi-r}, this is exactly $\mathcal{I}(f,\Psi ,r)$.
			\end{proof}
			\begin{remark}In the weight aspect, we have the similar transformation
				\[
				\int_0^{\infty}\psi(y) y^k |f(iy)|^2  \frac{dy}{y} \approx \sum_{|r| \ll k^{1/2}}\int_{Y_{0}(1)}P_{y\psi ,r}(z)y^{k}|f(z)|^{2}\frac{\dd x\dd y}{y^2}. 
				\]
				Note that in this case the shift $r$ is too large to bound each term trivially. We need to study the cancellation on average over $r$. Therefore, the problem is reduced to an average bound of shifted convolution sums (see also  \cite[Corollary 3.1 and Equ.(3.8)]{B-K-Y-fourth-norm}), which makes it more difficult than QUE in the weight aspect.

			\end{remark}

			\subsection{Estimates of $\mathcal{I}(f,\Psi ,r)$ for small $r$: Holowinsky's approach}For each normalized holomorphic newform $f$, we define
			\begin{equation}\label{M-f}
				M_{f}(x) = \frac{\prod\limits_{p\leq x}\left(1+\frac{2|\lambda_f(p)|}{p}\right)}{(\log ex)^{2}L(1,\ad f)}.
			\end{equation}
			Then Holowinsky's approach gives the following proposition.
			\begin{proposition}\label{Holowinsky's approach}For any $\varepsilon > 0$ and $ 0< |r| \leq (\log q)^{\varepsilon}$, we have
				\[
				\mathcal{I}(f,\Psi ,0) = 	\frac{3}{\pi}\int_{1}^{\infty}\psi(y)\frac{\dd y}{y} + O_{\psi,k}\left((\log q)^{\varepsilon}M_{f}(q)^{1/2}\right)
				\]
				and
				\[
				\mathcal{I}(f,\Psi ,r) = O_{\psi,k}\left( (\log qr)^{\varepsilon}M_{f}(qr) \right).
				\]
			\end{proposition}
			\begin{proof} The first part is Theorem 3.1 in \cite{Nelson-QUE-level}. For $ 0< |r| \leq (\log q)^{\varepsilon}$, we obtain the estimate from \eqref{Shifted-convolution-sum-f-Psi-r}
				\begin{equation}\label{Bounds-I-f-psi-small-r}
					\begin{aligned}
						\mathcal{I}(f,\Psi ,r)
						& \ll_{\psi,k}\frac{1}{L(1,\ad f)}\left(\sum_{n \leq \frac{qr}{2}}\frac{|\lambda_f(n)\lambda_f(n +qr)|}{qr} r^{-A} + \sum_{n \geq \frac{qr}{2}}\frac{|\lambda_f(n)\lambda_f(n +qr)|}{n} \left( \frac{n}{q}\right)^{-A}\right).
					\end{aligned}
				\end{equation}
				We have the following result in Nelson's work.
				\begin{lemma}[{\cite[Theorem 3.10]{Nelson-QUE-level}} ]Let $\varepsilon \in (0,1)$. Then for $x\geq 1$ and $\ell \in \mathbb{Z}_{\neq 0}$, we have
					\[
					\sum_{\substack{n\in\mathbb{N}\\ m = n+\ell\in \mathbb{N}\\ \max(n,m)\leq x}}|\lambda_f(n)\lambda_{f}(m)| \ll_{\varepsilon}\frac{x\prod\limits_{p\leq x}\left(1+\frac{2|\lambda_f(p)|}{p}\right)}{(\log ex)^{2-\varepsilon}},
					\]
					where all implied constants are absolute.
				\end{lemma}
				

				By partial summation, we get
				\[
				\mathcal{I}(f,\Psi ,r) \ll_{\psi,k}  (\log qr)^{\varepsilon}M_{f}(qr).
				\]
			\end{proof}
			\begin{remark}In fact, we have $M_{f}(q) \ll (\log q)^{\varepsilon}$. This alone is not sufficient to reach our goal.
				
			\end{remark}
			\subsection{Estimates of $\mathcal{I}(f,\Psi ,r)$ for small $r$: Soundararajan's approach}
			We need the weak subconvexity bound of $L$-functions introduced by Soundararajan.
			\begin{lemma}\label{Weak-subconvexity}Let $f$ be a holomorphic Hecke newform of weight $k$ fixed and square-free level $q$. Let $\phi$ be a Hecke--Maass cusp form of spectral parameter $t_\phi$ and level $1$.
				For any $\varepsilon >0$, we have
				\[
				L(\frac{1}{2}, \ad f \times \phi) \ll_{k,\varepsilon} \frac{t_\phi^{\frac{3}{2}+\varepsilon}q}{(\log q)^{1-\varepsilon}}.
				\]
				Moreover, we have
				\[
				L(\frac{1}{2}+it, \ad f ) \ll_{k,\varepsilon} \frac{t^{\frac{3}{4}}q^{\frac{1}{2}}}{(\log q)^{1-\varepsilon}}.
				\]
			\end{lemma}
			\begin{proof}See \cite[Theorem 1]{Sound-weak-subconvexity} and \cite[Theorem 1.3]{wattanawanichkul2024effectivecorrelationdecorrelationnewforms}.
			\end{proof}
			Let $\Lambda(s,\pi)$ be the complete $L$-function of $\pi$ (see \cite[Equ. (5.4)]{IK-ANT}). To relate the triple product formulas to the central value of $L$-functions, we need the following version of the Watson--Ichino formula.
			\begin{lemma}[{\cite[Theorem 4.1]{Nelson-QUE-level}}]\label{Watson-formula}
				Let $\phi$ be a Hecke--Maass cusp form of level $1$, and $f$ a holomorphic Hecke newform of square-free level $q$. Then 
				\[
				\frac{| \int_{Y_{0}(q)}\phi(z)y^{k}|f(z)|^{2}\frac{\dd x\dd y}{y^{2}}|^2}{ \int_{Y_{0}(1)}|\phi(z)|^{2}\frac{\dd x\dd y}{y^{2}} (\int_{Y_{0}(q)}y^{k}|f(z)|^{2}\frac{\dd x\dd y}{y^{2}})^{2} } = \frac{1}{8q}\frac{\Lambda(1/2, \phi)\Lambda(1/2, \ad f \times \phi)}{\Lambda(1,\ad \phi)\Lambda(1,\ad f)^2}.
				\]
			\end{lemma}
			We have the following proposition by using weak subconvexity.
			\begin{proposition}\label{Soundararajan's approach}For any $\varepsilon > 0$ and $ 0< |r| \leq (\log q)^{\varepsilon}$, we have
				\[
				\mathcal{I}(f,\Psi ,0) = 	\frac{3}{\pi}\int_{1}^{\infty}\psi(y)\frac{\dd y}{y} + O_{\psi,k}\left(\frac{1}{L(1,\ad f)(\log q)^{1-\varepsilon}}\right)
				\]
				and
				\[
				\mathcal{I}(f,\Psi ,r) = O_{\psi, k}\left( \frac{1}{L(1,\ad f)(\log q)^{1/2-\varepsilon}} \right).
				\]
			\end{proposition}
			\begin{proof}
				By definition, we have
				\[
				\mathcal{I}(f,\Psi ,0) = \frac{2\pi^2}{L(1,\ad f)}\frac{1}{8\pi^2 i}\int_{(3)}(\frac{4\pi}{q})^{-s}\frac{L(1+s,\ad f)\zeta^{(q)}(1+s)}{\zeta^{(q)}(2+2s)}\frac{\Gamma(s+k)}{\Gamma(k)}\int_{1}^{\infty}\psi(y)y^{-s}\frac{\dd y}{y} \dd s.
				\]
				Here $\zeta^{(q)}(s) := \zeta(s)\left(1 - \frac{1}{q^s}\right)$. After shifting the contour to the line $\Re(s) = -\frac{1}{2}$, the residue is
				\[
				\frac{3}{\pi}\int_{1}^{\infty}\psi(y)\frac{\dd y}{y}\times \frac{1}{1+q^{-1}}.
				\]
				The integration on the line $\Re(s) = -\frac{1}{2}$ is bounded by
				\[
				\frac{1}{L(1,\ad f)}\int_{\mathbb{R}}(\frac{4\pi}{q})^{1/2}\big|\frac{L(1/2+it,\ad f)\zeta^{(q)}(1/2+it)}{\zeta^{(q)}(1+2it)}\frac{\Gamma(k-1/2+it)}{\Gamma(k)}\int_{1}^{\infty}\psi(y)y^{1/2-it}\frac{\dd y}{y}\big| \dd t.
				\]
				Then 
				\begin{equation}\label{I-f-psi-0-main-term}
					\begin{aligned}
						\mathcal{I}(f,\Psi ,0) &= \frac{3}{\pi}\int_{1}^{\infty}\psi(y)\frac{\dd y}{y} + O_{\psi,k}\left(\frac{\int_{-\infty}^{+\infty}e^{-|t|}|L(\frac{1}{2}+it,\ad f)|\dd t}{q^{1/2}L(1,\ad f)}\right)\\
						& = \frac{3}{\pi}\int_{1}^{\infty}\psi(y)\frac{\dd y}{y} + O_{\psi,k,\varepsilon}\left(\frac{1}{L(1,\ad f)(\log q)^{1-\varepsilon}}\right).
					\end{aligned}
				\end{equation}
				For $ 0< |r| \leq (\log q)^{\varepsilon}$, by the spectral decomposition, we have
				\begin{multline}\label{Spectral-decomposition}
					\int_{Y_{0}(q)}P_{y\psi ,r}(z)y^{k}|f(z)|^{2}\frac{\dd x\dd y}{y^2} = \sum_{j\geq 1}\langle P_{y\psi , r} , \phi_j\rangle_{1}	\int_{Y_{0}(q)}\phi_j(z)y^{k}|f(z)|^{2}\frac{\dd x\dd y}{y^2}\\ + \frac{1}{4\pi}\int_{\mathbb{R}}\langle P_{y\psi , r} , E_{t}\rangle_{1}	\int_{Y_{0}(q)}E_{t}(z)y^{k}|f(z)|^{2}\frac{\dd x\dd y}{y^2}\dd t.
				\end{multline}
				Let $\rho_j(1)$ be the first Fourier coefficient of $\phi_j$ and $|\rho_j(1)|^2 = \frac{\cosh \pi t_{j}}{2L(1,\ad \phi_j)}$. The Hecke eigenvalue of $\phi_j$ satisfies $\lambda_j(r) \ll r^{1/2}$ from Hecke's bound. By the unfolding trick and Mellin inversion, we have
				\begin{equation}
					\begin{aligned}
						\langle P_{y\psi , r} , \phi_j\rangle_{1} &= 2\rho_j(1)\lambda_j(r)\int_{0}^{\infty}\psi(y)y^{1/2}K_{it_j}(2\pi|r| y)\dd y\\
						& = 2\rho_j(1)\lambda_j(r) \frac{1}{2\pi i}\int_{(\sigma)}2^{s-2}\Gamma(\frac{s+it_j}{2})\Gamma(\frac{s-it_j}{2})\widetilde{\psi}(-s-1/2) |r|^{-s} \dd s\\
						& \ll_{\psi} r^{1/2}(\frac{1+|t_j|}{r})^{-A}.
					\end{aligned}
				\end{equation}
				To obtain the final estimate, we shift the contour to the line $\Re(s) = -A$ and use Stirling's formula. The contribution at the poles of $\Gamma\left(\frac{s+it_j}{2}\right)$ can be controlled by $\widetilde{\psi}(-\sigma+it_j)\ll t_{j}^{-A}$. Similarly, we get $\langle P_{y\psi , r} , E_{t}\rangle_{1} \ll r^{1/2}(\frac{1+|t|}{r})^{-A}.$
				Thus, we have
				\[
				\sum_{t_j\geq (\log q)^{4\varepsilon}}\langle P_{y\psi , r} , \phi_j\rangle_{1}	\int_{Y_{0}(q)}\phi_j(z)y^{k}|f(z)|^{2}\frac{\dd x\dd y}{y^2} \ll_{\psi,k,\varepsilon} (\log q)^{-2026} 
				\]
				from the Watson--Ichino formula and the weak subconvexity bound. Similarly, we have
				\[
				\frac{1}{4\pi}\int_{|t| \geq (\log q)^{4\varepsilon}}\langle P_{y\psi , r} , E_{t}\rangle_{1}	\int_{Y_{0}(q)}E_{t}(z)y^{k}|f(z)|^{2}\frac{\dd x\dd y}{y^2}\dd t \ll_{\psi,k,\varepsilon} (\log q)^{-2026}. 
				\]
				In conclusion, we obtain
				\begin{multline*}
					\int_{Y_{0}(q)}P_{y\psi ,r}(z)y^{k}|f(z)|^{2}\frac{\dd x\dd y}{y^2} \\
					\ll_{\psi, k} (\log q)^{\varepsilon} \big(\max_{\substack{t_j \leq (\log q)^{4\varepsilon} \\ t \leq (\log q)^{4\varepsilon}}}\frac{L(\frac{1}{2}, \ad f \times \phi_j)^{1/2}}{q^{1/2}L(1,\ad f)} + \frac{L(\frac{1}{2}+it, \ad f )}{q^{1/2}L(1,\ad f)}\big) 
					\ll \frac{1}{L(1,\ad f)(\log q)^{1/2-\varepsilon}}.
				\end{multline*}
				
			\end{proof}
			\subsection{End of the proof of Theorem \ref{Theorem-psi-all-geodesic}: optimal process}
			From Proposition \ref{Holowinsky's approach} and Proposition \ref{Soundararajan's approach}, we have
			\begin{multline}
				\mathcal{I}(f,\Psi) - \frac{3}{\pi}\int_{1}^{\infty}\psi(y)\frac{\dd y}{y}  \ll_{\psi, k, \varepsilon} 	 	 (\log q)^{\varepsilon}\min\big\{M_{f}(q(\log q)^{\varepsilon})^{1/2},  \frac{1}{L(1,\ad f)(\log q)^{1/2-\varepsilon}}\big\}
			\end{multline}
			where $M_{f}(x)$ is defined by \eqref{M-f}.
			\begin{lemma}We have
				\[
				M_{f}(q(\log q)^{\varepsilon}) \ll (\log q)^{\varepsilon}
				\]
				and
				\[
				M_{f}(q(\log q)^{\varepsilon}) \ll (\log q)^{1/6+\varepsilon}L(1,\ad f)^{1/2}.
				\]
			\end{lemma}
			\begin{proof}Using the bound (see \cite[Lemma 2]{Holowinsky-Sound-QUE-holo}, \cite[Lemma 7.2]{wattanawanichkul2024effectivecorrelationdecorrelationnewforms})
				\[
				L(1,\ad f)^{-1} \ll_{k,\varepsilon} (\log\log q)^{9}\prod\limits_{ p \leq kq}\left(1 - \frac{\lambda_{f}(p)^2 - 1}{p}\right)
				\]
				and $\prod\limits_{p\leq x}\left(1+\frac{a}{p}\right) \asymp (\log x)^{a}$. We get
				\[
				M_{f}(q(\log q)^{\varepsilon}) \ll_{k, \varepsilon} (\log q)^{\varepsilon}\prod\limits_{p\leq q(\log q)^{\varepsilon}}\left(1 - \frac{(|\lambda_{f}(p)|-1)^2}{p}\right)\ll_{k, \varepsilon}(\log q)^{\varepsilon}.
				\]
				For the second one, we have
				\begin{equation}
					\begin{aligned}
						\prod\limits_{p\leq q(\log q)^{\varepsilon}}\left(1+\frac{2|\lambda_f(p)|}{p}\right) &\leq \prod\limits_{p\leq q(\log q)^{\varepsilon}}\left(1+ \frac{2}{3p} + \frac{3}{2} \frac{\lambda_f(p)^2}{p}\right)\\
						& \ll \prod\limits_{p\leq q(\log q)^{\varepsilon}}\left(1+ \frac{13}{6p}\right)\left(1 + \frac{3}{2}\frac{\lambda_f(p)^2 - 1}{p}\right)\\
						& \ll (\log q)^{13/6+\varepsilon}L(1, \ad f)^{3/2}.
					\end{aligned}
				\end{equation}
				Here we may use $\prod\limits_{kq\leq p\leq q(\log q)^{\varepsilon}}\left(1 + \frac{a}{p}\right) \ll_{k, \varepsilon} \log\log q$. This completes the proof.
			\end{proof}
			If $L(1 ,\ad f) \leq (\log q)^{\delta}$ and $\delta \geq -1$, the error term is bounded by
			\[
			\max\limits_{\delta \geq -1}\min\big\{(\log q)^{1/12+\delta/4} , (\log q)^{-1/2-\delta}\big\} = (\log q)^{-1/30}.
			\]
			The maximum is attained at $\delta=-\frac{7}{15}$. In conclusion, this proves Theorem \ref{Theorem-psi-all-geodesic}.
			\section{\label{sec:4}The case of square-free levels}
			Essentially, we only need to explain how the key Lemma \ref{Lemma-incomplete-Poincare-series} holds. We may study the periods with incomplete Poincar\'e series as we do in \S \ref{sec:3}.
			
			For $q = p_1p_2\cdots p_r$ an odd square-free number, we choose the special coset representatives
			\begin{equation}\label{coset-representatives-square-free-q}
				\Gamma_0(q)\backslash\Gamma(1) = \big\{\begin{pmatrix}
					1 &0\\
					0 &1
				\end{pmatrix}\big\} \cup  \big\{ \begin{pmatrix}
					0 &-1\\
					1 &-\ell
				\end{pmatrix}, 0\leq|\ell|\leq \frac{q-1}{2}\big\}\cup_{\substack{d \mid q\\ 1 < d < q }} Q_{d}
			\end{equation}
			where 
			\begin{equation}
				Q_{d} =  \big\{ \begin{pmatrix}
					0 &-1\\
					1 &-d
				\end{pmatrix}\begin{pmatrix}
					0 &-1\\
					1 &-\ell_d
				\end{pmatrix}, 0\leq|\ell_d|\leq \frac{q/d-1}{2}\big\}.
			\end{equation}
			Note that $Q_d$ and the second set correspond to the cusps $\frac{1}{d}$ and $0$, respectively. The Atkin--Lehner operators for $\Gamma_0(q)$ can be chosen from the set
			\begin{equation}
				\big\{ \begin{pmatrix}
					0 &-1\\
					q & 0
				\end{pmatrix}\big\}\cup\big\{ \begin{pmatrix}
					1 &  0\\
					-q   &  q/d
				\end{pmatrix}, d | q, 1<d<q \big\}.
			\end{equation}
			We define 
			\[
			\omega_{d} =\begin{pmatrix}
				1 &  0\\
				-q   &  q/d
			\end{pmatrix}
			\]
			for $d | q, 1<d<q$. In fact, the actual Atkin--Lehner operator 
			\[
			\omega_{d}' = \begin{pmatrix}
				xq/d  &  y\\
				zq   &  kq/d 
			\end{pmatrix} = \gamma\omega_{d}
			\]
			where $\gamma \in \Gamma_0(q)$. Then for $f \in H_{k}^{*}(q)$, we get
			\[
			f_{{|}_k \omega_{d}} = f_{{|}_k \omega_{d}'} = \pm f
			\]
			where the action of slash operator is $f_{{|}_k A}(z) = (\det A)^{k/2}(cz+d)^{-k}f(z)$ for $A =  \begin{pmatrix}
				a  &  b\\
				c   &  d 
			\end{pmatrix}  \in \GL(2,\mathbb{R})^{+}$.

			By the parameterization of $\mathcal{L},\mathcal{L}_d$ as in \S \ref{sec:2}, we get
			\begin{equation}
				\begin{aligned}
					\mathcal{I}(f,\Psi) &= \int_{1}^{\infty}\psi(y)y^{k}f(iy)^2\frac{\dd y}{y}\\
					& + \sum_{0 \leq |\ell| \leq \frac{q-1}{2}}\int_{1}^{\infty}\psi\left(y\right)\left(\frac{y}{\ell^2 + y^2}\right)^{k}\big|f\left(\frac{\ell}{\ell^2 +y^2}+ i\frac{y}{\ell^2 + y^2}\right)\big|^2 \frac{\dd y}{y}\\
					& + \sum_{\substack{d |q\\1<d<q}}\sum_{0 \leq |\ell_d| \leq \frac{q/d-1}{2}}\int_{1}^{\infty}\psi\left(y\right)\left(\Im (\gamma_d \cdot iy)\right)^{k}\big|f\left( \gamma_d \cdot iy\right)\big|^2 \frac{\dd y}{y},
				\end{aligned}
			\end{equation}
			where $ \gamma_d = \begin{pmatrix}
				0 &-1\\
				1 &-d
			\end{pmatrix}\begin{pmatrix}
				0 &-1\\
				1 &-\ell_d
			\end{pmatrix}$. By Atkin--Lehner theory, we have
			\[
			\det(\omega_{d})^{k/2}(-qz+q/d)^{-k}f(\omega_d\cdot z) = \pm f(z).
			\]
			Taking $z = \gamma_d \cdot iy$ gives
			\[
			\big|f\left( \gamma_d \cdot iy\right)\big|^2 = \left(\frac{q}{d}\right)^{-k}|dz - 1|^{-2k} \big|f\left(\frac{-\ell_d + iy}{q/d}\right)\big|^2.
			\]
			From $\gamma_d \cdot iy = \frac{\ell_d(d\ell_d - 1)+dy^2 + iy}{(d\ell_d-1)^2 + (dy)^2}$ and $|dz - 1|^{2} = \frac{1}{(d\ell_d-1)^2 + (dy)^2}$, we get
			\begin{equation}
				\left(\Im (\gamma_d \cdot iy)\right)^{k}\big|f\left( \gamma_d \cdot iy\right)\big|^2 = \left(\frac{y}{q/d}\right)^k\big|f\left(\frac{-\ell_d + iy}{q/d}\right)\big|^2.
			\end{equation}
			We also have 
			\begin{equation}
				\left(\frac{y}{\ell^2 + y^2}\right)^{k}\big|f\left(\frac{\ell}{\ell^2 +y^2}+ i\frac{y}{\ell^2 + y^2}\right)\big|^2 = \left(\frac{y}{q}\right)^k\big|f\left(\frac{-\ell + iy}{q}\right)\big|^2.
			\end{equation}
			
			Hence, we establish \eqref{Vertical-lines}. After applying the Fourier expansion and truncating the large shifts, the key Lemma \ref{Lemma-incomplete-Poincare-series} still holds from the expression \eqref{Unfolding-trick}.
			
			If $2 | q$, for $d$ with $2\nmid d$, we replace $Q_d$ by the following set:
			\begin{equation}
				P_{d} =  \big\{ \begin{pmatrix}
					0 &-1\\
					1 &-d
				\end{pmatrix}\begin{pmatrix}
					0 &-1\\
					1 &-\ell_d
				\end{pmatrix}, 0 < |\ell_d|\leq \frac{q/d}{2}\big\}.
			\end{equation}

			\section{\label{sec:5}The proof of Theorem \ref{Theorem-1-all-geodesic}}
			Let $q$ be prime.	We deal with $\mathcal{I}(f,1)$ as \eqref{Fourier-expansion-I-f-Psi} and get
			\begin{equation}
				\begin{aligned}
					\mathcal{I}&(f,1) = \frac{2\pi^2}{L(1,\ad f)q}\frac{1}{4\pi^2 i} \times\\ &\int_{(3)}(2\pi)^{-s}\sum_{t \geq 2}\frac{\sum_{n+m = t}\lambda_{f}(n)\lambda_{f}(m)\left(\frac{2\sqrt{nm}}{n+m}\right)^{k-1}\left(\frac{1}{s} + \frac{q^{s}}{s}\delta(n\equiv m \bmod q)\right)}{t^{1+s}}\frac{\Gamma(s+k)}{\Gamma(k)} \dd s.
				\end{aligned}
			\end{equation}
			The contribution of the term containing $\frac{1}{s}$ is bounded by $O_{k,\varepsilon}(q^{-1+\varepsilon})$ from Deligne's bound. The diagonal term over $n = m$ is
			\[
			\frac{2\pi^2}{L(1,\ad f)}\frac{1}{8\pi^2 i}\int_{(3)}(\frac{4\pi}{q})^{-s}\frac{L(1+s,\ad f)\zeta^{(q)}(1+s)}{\zeta^{(q)}(2+2s)}\frac{\Gamma(s+k)}{\Gamma(k)} \frac{\dd s}{s}.
			\]
			By Laurent's expansion, we have
			\begin{equation}
				\begin{aligned}
					\left(\frac{4\pi}{q}\right)^{-s} &= 1 - s\log \frac{4\pi}{q} +O(s^2),\\
					L(1+s, \ad f) &= L(1,\ad f) + sL'(1,\ad f) +O(s^2),\\
					\frac{1}{\zeta^{(q)}(2+2s)} &= \frac{1}{\zeta^{(q)}(2)} - 2s \frac{\zeta^{(q)'}(2)}{\left(\zeta^{(q)}(2)\right)^2} + O(s^2),\\
					\frac{\Gamma(s+k)}{\Gamma(k)}& = 1 + s\frac{\Gamma'(k)}{\Gamma(k)} +O(s^2),\\			
					\frac{\zeta(1+s)}{s}&= \frac{1}{s^2} + \frac{\gamma}{s}+O(1).
				\end{aligned}
			\end{equation}
			After shifting the contour to the line $\Re(s) = -\frac{1}{2}$, the residue is 
			\begin{multline*}
				\frac{\gamma L(1,\ad f)}{\zeta^{(q)}(2)} +
				\left(\frac{(- \log \frac{4\pi}{q})L(1,\ad f)}{\zeta^{(q)}(2)} + \frac{L'(1,\ad f)}{\zeta^{(q)}(2)} - \frac{2\zeta^{(q)'}(2)L(1,\ad f)}{\left(\zeta^{(q)}(2)\right)^2} + \frac{L(1,\ad f)}{\zeta^{(q)}(2)}\frac{\Gamma'(k)}{\Gamma(k)}\right).
			\end{multline*}
			Note that $\zeta^{(q)}(2) =\zeta(2)(1-q^{-2}) = \frac{\pi^2}{6}(1-q^{-2})$ and $\frac{\zeta^{(q)'}(2)}{\zeta^{(q)}(2)} = \frac{\zeta'(2)}{\zeta(2)} + \frac{q^{-2}\log q}{1-q^{-2}}$. The contribution of the residue is
			\[
			\frac{3}{\pi}\log q + \frac{3}{\pi}\frac{L'(1,\ad f)}{L(1,\ad f)} + \frac{3}{\pi}\left(  \gamma -\log4\pi -2 \frac{\zeta'(2)}{\zeta(2)} +  \frac{\Gamma'(k)}{\Gamma(k)}   \right) + O(q^{-1}).
			\]
			The integral on the line $\Re(s) = -\frac{1}{2}$ is bounded by
			\begin{equation}
				O_{k}\left(\frac{\int_{-\infty}^{+\infty}e^{-|t|}|L(\frac{1}{2}+it,\ad f)|\frac{\dd t}{1+|t|}}{q^{1/2}L(1,\ad f)}\right) = O_{k,\varepsilon}\left((\log q)^{\varepsilon}\right).
			\end{equation}
			The remainder terms with $n \neq m$ are bounded by $M_{f}(q(\log q)^{\varepsilon}) = O_{k, \varepsilon}\left((\log q)^{\varepsilon}\right)$. This completes the proof of Theorem \ref{Theorem-1-all-geodesic}.

			\section{\label{sec:6}The dual geodesics and relations to twisted $L$-functions}
			In this section, we give the full details of the proof of Corollary \ref{Corollary-twisted}. We use the same coset representatives of $\Gamma_0(q)\backslash\Gamma(1)$ as \eqref{coset-representatives}. Let $\mathcal{L}_\ell = \begin{pmatrix}
				0 &-1\\
				1 &-\ell
			\end{pmatrix}\mathcal{L}$. We define
			\[
			\mathcal{I}'(f, 1) =  \int_{	 \mathcal{L}}y^{k}|f(z)|^2\dd s(z),\quad\quad	 \mathcal{I}_{\ell}(f, 1) =  \int_{	 \mathcal{L}_{\ell}}y^{k}|f(z)|^2\dd s(z) 
			\]
			for $0\leq|\ell|\leq \frac{q-1}{2}$. Therefore, we have $\mathcal{I}(f,1) =  \mathcal{I}'(f, 1) + \sum_{0\leq|\ell|\leq \frac{q-1}{2}}\mathcal{I}_{\ell}(f, 1) $.
			We now introduce the notion about the dual geodesics.
			\begin{definition}Let $\mathcal{L}_\ell = \begin{pmatrix}
					0 &-1\\
					1 &-\ell
				\end{pmatrix}\mathcal{L}$. For $0<|\ell|\leq \frac{q-1}{2}$, we get the unique integer $0<|\overline{\ell}|\leq \frac{q-1}{2}$ such that $\ell \overline{\ell} \equiv 1 \mod q$. We call  $\mathcal{L}_{\overline{\ell}}$ the dual  geodesic of $\mathcal{L}_{\ell}$. In particular, $\mathcal{L}$ is defined to be the dual of $\mathcal{L}_0$.
			\end{definition}
			By the same method in Section \ref{sec:2}, we have
			\begin{equation}\label{I-ell}
				\mathcal{I}_{\ell}(f,1) = \int_{1}^{\infty}\left(\frac{y}{q}\right)^{k}\big|f\left(\frac{-\ell + iy}{q}\right)\big|^2 \frac{\dd y}{y} = \int_{1/q}^{\infty}y^{k}\big|f\left(\frac{-\ell}{q} + iy\right)\big|^2 \frac{\dd y}{y}
			\end{equation}
			for $0\leq |\ell| \leq \frac{q-1}{2}$.
			For $0<|\ell|\leq \frac{q-1}{2}$, the cusp $-\frac{\ell}{q}$ is equivalent to $\infty$. We choose $\gamma = \begin{pmatrix}
				-\ell & b\\
				q   &  d
			\end{pmatrix} \in \Gamma_0(q)$ such that $-\ell d \equiv 1 \bmod q$ and $\gamma (\infty) = -\frac{\ell}{q}$. We get
			\[
			\int_{0}^{1/q}y^{k}\big|f\left(\frac{-\ell}{q} + iy\right)\big|^2 \frac{\dd y}{y} = \int_{1/q}^{\infty}y^{k}\big|f\left(\frac{d}{q} + iy\right)\big|^2 \frac{\dd y}{y} = \int_{1/q}^{\infty}y^{k}\big|f\left(\frac{-\overline{\ell}}{q} + iy\right)\big|^2 \frac{\dd y}{y}
			\]
			from $y^{k}|f(z)|^2 = (\Im\gamma^{-1}z)^{k}|f(\gamma^{-1}z)|^2$.
			Consequently,
			\[
			\mathcal{I}_{\ell}(f,1) + \mathcal{I}_{\overline{\ell}}(f,1) =  \int_{0}^{\infty}y^{k}\big|f\left(\frac{-\ell}{q} + iy\right)\big|^2 \frac{\dd y}{y}.
			\]
			This explains why we call these two geodesic branches dual.	By Parseval's identity, we get
			\begin{equation}\label{Parseval}
				\int_{0}^{\infty}y^{k}\big|f\left(\frac{-\ell}{q} + iy\right)\big|^2 \frac{\dd y}{y} = \frac{|a_{f}(1)|^2}{(2\pi)^{k+1}}\int_{-\infty}^{+\infty}|\Gamma(\frac{k}{2}+it)|^2 |L(\frac{1}{2}+it, -\frac{\ell}{q}, f)|^2 \dd t
			\end{equation}
			where $L(s , -\frac{\ell}{q} ,f)$ is the additive twisted $L$-function. Then we get
			\begin{equation}
				\begin{aligned}
					2\mathcal{I}(f,1) 
					& = 2\int_{0}^{\infty}y^{k}|f(iy)|^2 \frac{\dd y}{y} +\sum_{0 < |\ell| \leq \frac{q-1}{2}}\int_{0}^{\infty}y^{k}|f(\frac{\ell}{q}+iy)|^2 \frac{\dd y}{y}\\
					& = \sum_{0\leq|\ell|\leq \frac{q-1}{2}}\frac{|a_{f}(1)|^2}{(2\pi)^{k+1}}\int_{-\infty}^{+\infty}|\Gamma(\frac{k}{2}+it)|^2 |L(\frac{1}{2}+it, -\frac{\ell}{q}, f)|^2 \dd t.
				\end{aligned}
			\end{equation}
			Thus we have established \eqref{I-f-1-moments-L-function} now.
			\begin{remark}Similarly, one obtains
				\begin{equation}\label{Psi-integration-L-function}
					2\mathcal{I}(f,\Psi) = 2\int_{0}^{\infty}\psi(y)y^{k}|f(iy)|^2 \frac{\dd y}{y} +\sum_{0 < |\ell| \leq \frac{q-1}{2}}\int_{0}^{\infty}\psi(qy)y^{k}|f(\frac{\ell}{q}+iy)|^2 \frac{\dd y}{y}.
				\end{equation} 
			\end{remark}
			\subsection{Relations between additive twists and multiplicative twists}
			We recall the basic relations between additive and multiplicative characters. Let the Gauss sum be
			\[
			\tau(\chi) := \sum_{a\bmod m}\chi(a)e(\frac{a}{m})
			\]
			for a character $\chi \bmod m$. We have
			\begin{equation}\label{Character-transforms}
				\begin{aligned}
					e(\frac{a}{m}) &= \frac{1}{\varphi(m)}\sum_{\chi \bmod m}\overline{\chi}(a)\tau(\chi),\quad\quad \gcd(a,m) = 1,\\
					\chi(a)\tau(\overline{\chi}) &= \sum_{b\bmod m}\overline{\chi}(b)e(\frac{ab}{m}).
				\end{aligned}
			\end{equation}
			The second identity holds for primitive characters $\chi$ modulo $m$. For $q$ prime and $0<|\ell|\leq \frac{q-1}{2}$, we define the Dirichlet series
			\begin{equation}\label{Additive-twisted-L-function}
				L(s,\frac{\ell}{q},f) :=\sum_{n\geq1} \frac{\lambda_{f}(n)e(\frac{\ell n}{q})}{n^s},\quad\quad \Re(s)>1.
			\end{equation}
			For $\Re(s)$ sufficiently large, we have
			\[
			L(s,\frac{\ell}{q},f) =\sum_{\substack{n\geq1\\ \gcd(n,q)=1}} \frac{\lambda_{f}(n)e(\frac{\ell n}{q})}{n^s} + \sum_{\substack{n\geq1}} \frac{\lambda_{f}(qn)}{(qn)^s}.
			\]
			By \eqref{Character-transforms}, we get
			\[
			\sum_{\substack{n\geq1\\ \gcd(n,q)=1}} \frac{\lambda_{f}(n)e(\frac{\ell n}{q})}{n^s} = \frac{1}{\varphi(q)}\sum_{\chi \bmod q}\chi(\ell)\tau(\overline{\chi})\sum_{n \geq 1}\frac{\lambda_{f}(n)\chi(n)}{n^{s}} = \frac{1}{\varphi(q)}\sum_{\chi \bmod q}\chi(\ell)\tau(\overline{\chi})L(s,\chi,f)
			\]
			where the multiplicative twisted $L$-function $L(s,\chi,f)$ is defined by \eqref{Multiplicative-twisted}. 
			\begin{remark}See \cite[Appendices A ]{K-M-V-Rankin-Selberg-level-aspect} and the theory of standard $L$-functions $L(s,f)$. Both $ L(s,\frac{\ell}{q},f)$ and $\sum\limits_{\substack{n\geq1}} \frac{\lambda_{f}(qn)}{(qn)^s} = \frac{\lambda_f(q)}{q^{s}}L(s,f)$ admit analytic continuation to entire functions in $\mathbb{C}$. Then $\sum\limits_{\substack{n\geq1\\ \gcd(n,q)=1}} \frac{\lambda_{f}(n)e(\frac{\ell n}{q})}{n^s}$ have the analytic continuation as well. By an inversion formula, the same holds for $L(s,\chi,f)$.
			\end{remark}
			Now we get
			\begin{multline}
				\sum_{0<|\ell|\leq \frac{q-1}{2}}| L(s,\frac{\ell}{q},f)|^2 = \frac{1}{\varphi(q)}\sum_{\chi \bmod q}|\tau(\overline{\chi})|^2| L(s,\chi,f)|^2\\ + 2\Re \{ \overline{\tau(\chi_0)L(s,\chi_0,f)} \sum_{\substack{n\geq1}} \frac{\lambda_{f}(qn)}{(qn)^s}\} + |\sum_{\substack{n\geq1}}\frac{\lambda_{f}(qn)}{(qn)^s}|^2. 
			\end{multline}
			In particular, we have the following proposition.
			\begin{proposition}\label{Additive-multiplicative-L-function}Assume that $q$ is prime and $t \ll q^{\varepsilon}$. Let $f$ be a Hecke newform of weight $k$ and level $q$. Then
				\begin{equation}
					\sum_{0<|\ell|\leq \frac{q-1}{2}}| L(\frac{1}{2}+it,\frac{\ell}{q},f)|^2 = \frac{q}{\varphi(q)}{\sum_{\chi \bmod q}}^{*}| L(\frac{1}{2}+it,\chi,f)|^2  + O_{k,\varepsilon}(q^{-1/2+\varepsilon}).
				\end{equation}
			\end{proposition}   
			Replacing the left side of \eqref{Additive-multiplicative-moments} by the multiplicative twisted $L$-function completes the proof of Corollary \ref{Corollary-twisted}.
			
			\section*{Acknowledgements}
			
			The author would like to thank Prof. Bingrong Huang for his encouragement and helpful discussions. He also gratefully thanks the referees for
			their constructive
			comments.

			\newcommand{\etalchar}[1]{$^{#1}$}

			
			
			

		\end{document}